\documentclass[12pt, reqno]{amsart}

\usepackage[T1]{fontenc}
\usepackage[utf8]{inputenc}
\usepackage{geometry}
\usepackage{amsmath,amssymb,amsthm,mathtools, float}
\usepackage{bm}
\usepackage[foot]{amsaddr}
\usepackage{booktabs}
\usepackage[colorlinks=true,linkcolor=blue,citecolor=blue,urlcolor=blue]{hyperref}

\numberwithin{equation}{section}

\newtheorem{theorem}{Theorem}[section]
\newtheorem{lemma}[theorem]{Lemma}
\newtheorem{proposition}[theorem]{Proposition}
\newtheorem{corollary}[theorem]{Corollary}
\theoremstyle{definition}

\theoremstyle{remark}

\newcommand{\R}{\mathbb{R}}
\newcommand{\N}{\mathbb{N}}
\newcommand{\Z}{\mathbb{Z}}
\newcommand{\E}{\mathbb{E}}
\newcommand{\PP}{\mathbb{P}}
\newcommand{\F}{\mathcal{F}}
\newcommand{\G}{\mathcal{G}}
\newcommand{\calH}{\mathcal{H}}
\newcommand{\calM}{\mathcal{M}}
\newcommand{\calL}{\mathcal{L}}
\newcommand{\TV}{\mathrm{TV}}
\newcommand{\Var}{\operatorname{Var}}
\newcommand{\Unif}{\operatorname{Unif}}
\newcommand{\1}{\mathbf{1}}
\usepackage[backend=biber,style=numeric,sorting=nyt]{biblatex}
\makeatletter
\renewcommand{\email}[2][]{%
  \ifx\emails\@empty\relax\else{\g@addto@macro\emails{,\space}}\fi%
  \@ifnotempty{#1}{\g@addto@macro\emails{\textrm{(#1)}\space}}%
  \g@addto@macro\emails{#2}%
}

\title{Coordinate-wise Elephant Random Walk}
\author{Denisse Alejandra Escobar Parra}
\address{Department of Statistics and Applied Probability, University of California, Santa Barbara}
\email{descobarparra@ucsb.edu}

\begin{document}

\begin{abstract}
We introduce a coordinate-wise version of the elephant random walk on the $k$-dimensional discrete hypercube $Q_k=\{-1,1\}^k$. At each global time step, one coordinate is selected uniformly at random and updated according to an elephant-type memory rule using only the past values of that coordinate. The resulting process is a nearest-neighbor walk with possible holding on the hypercube, but it is not Markovian on $Q_k$ because the transition probabilities depend on coordinate-wise empirical histories. We show that, when all memory parameters satisfy $p_i<1$, the coordinate-wise memory biases vanish almost surely. Consequently, the time-dependent transition kernels of the walk are asymptotically close in total variation to the memoryless coordinate-refresh kernel. Using this perturbation argument and the uniform ergodicity of the refresh chain, we prove a weak law of large numbers for bounded observables. We then study the Doob martingale associated with a bounded observable and prove a martingale central limit theorem and functional central limit theorem. The limiting variance is determined by the refresh kernel and the uniform measure on the hypercube, and is completely independent of the memory parameters $p_1,p_2,\dots,p_k.$ 
\end{abstract}

\maketitle
\section{Introduction}
The Elephant Random Walk (ERW), introduced by Sch\"utz and Trimper \cite{Schutz2004}, is a discrete-time process on $\Z$ with complete long-range memory: at each step the walker selects a past step uniformly at random and repeats it with probability $p$ or reverses it with probability $1-p.$ Despite its simple definition, the ERW exhibits a sharp behavioral trichotomy: diffusive for $p<\frac{3}{4},$ critical at $p=\frac{3}{4}$ and superdiffusive for $p>\frac{3}{4}.$

The rigorous mathematical analysis of the ERW has been advanced through several complementary approaches. Coletti, Gava and Sch\"utz \cite{Coletti2017} established central limit theorems in the diffusive and critical regimes and showed that in the superdiffusive case the rescaled position converges to a non-degenerate random variable. Bercu \cite{Bercu2018} identified a natural martingale in the model and used it to derive convergence rates, central limit theorems and laws of iterated logarithm; this martingale perspective has since become a standard tool. Baur and Bertoin \cite{Baur2016} connected the ERW to P\'olya-type urns, placing it within the broader theory of reinforcement processes and proving functional central limit theorems. More recently, Tokumitsu \cite{Tokumitsu2026} established stable functional central limit theorems in all three regimes. A structural feature shared across all of these works is that memory acts by sampling from the full history of past increments.

The multi-dimensional setting has been studied by Bercu and Laulin \cite{BercuLaulin2019} who extended the martingale approach to walks on $\Z^k$ where memory acts on full vector-valued increments and by  Bertenghi \cite{Bertenghi2022}, who obtained functional limit theorems via P\'olya urns. Geometric extensions appear in \cite{Shibata2025}, who treated general periodic structures, and Chaudhuri \cite{Chaudhuri2026}, who studied the triangular lattice. Further variants are analyzed in \cite{LaulinMallein2025} and  \cite{GhoshDhillonKataria2026}.

The present paper introduces a qualitatively different variant: the {\it Coordinate-wise Elephant Random Walk} (CERW) on the discrete hypercube $Q_k=\{-1,1\}^k.$ The key novelty is that memory acts {\it independently per coordinate:} at each step a coordinate $i$ is chosen uniformly at random, and its new value is drawn from the past values of coordinate $i$ alone, then repeated or reversed with probability $p_i.$ All other coordinates remain unchanged. This distinguishes the CERW from all previous multi-dimensional ERWs in two essential ways. First, the space $Q_k$ is {\it finite,} so diffusive rescaling of the global positions is not possible. Second, the memory is {\it anisotropic:} each coordinate carries its own history, and the interaction between these coordinate-wise dynamics and the geometry of the hypercube is the central object of study. 

The analysis is carried out in three steps. First, we show that when $p_i<1$ for all $i,$ the empirical history bias $\widehat{m}_n^{(i)}$ of each coordinate converges to zero almost surely. Second, we compare the resulting long-time dynamics with the memoryless refresh kernel $K^{\star},$ in which each coordinate is independently set to $\pm1$ with equal probability. A key technical contribution is Lemma \ref{het_lem:finite-horizon}, which uses an inductive argument to handle the path-dependence of the random kernels. The comparison yields a weak law of large numbers: for every bounded function $f:Q_k\to\R$, the time average $n^{-1}\sum_{t}f(S_t)$ converges in probability to $\pi(f),$ the expectation under the uniform measure $\pi.$ Third, we study the Doob martingale $M_n^{(f)}$ associated with $f.$ Its predictable quadratic variation satisfies $n^{-1}\langle M^{(f)}\rangle_n\to\sigma^2_f:=\pi(v_f) $ in probability, as $n \to \infty$, and a martingale central limit theorem and functional CLT follow. 

A striking feature of all limit theorems is {\it universality of the limiting variance:} $\sigma_f^2$ depends only on the observable $f$ and the geometry of the hypercube through $K^{\star}$ and $\pi,$  and it is completely independent of the memory parameters $p_1,p_2,\dots, p_k.$ In the last section we compute $\sigma^2_f$ explicitly for eight natural observables on the hypercube: height, Hamming distance, parity, single coordinates, pair correlations, level indicators, vertex indicators and face indicators. 

The remainder of the paper is organized as follows. In Section \ref{sec:2}, we introduce the Coordinate-wise Elephant Random Walk (CERW) on the hypercube $Q_k$ and its coordinate-wise update rule. Section \ref{sec:3} studies the one-coordinate dynamics in isolation, showing via a stochastic approximation argument that each coordinate's empirical memory bias vanishes almost surely whenever $p_i<1.$ In Section \ref{sec:4}, we use this fact to compare the random transition kernel of the CERW with the memoryless refresh kernel $K^{\star},$  obtaining a total variation bound (Lemma \ref{het_lem:tv-bound}) and a finite-horizon comparison estimate (Lemma \ref{het_lem:finite-horizon}) that control the discrepancy between the two dynamics. Section \ref{sec:5} combines this perturbative comparison with the uniform ergodicity of $K^{\star}$ to establish a weak law of large numbers for bounded observables (Theorem \ref{het_thm:WLLN}). In Section \ref{sec:6}, we construct the Doob martingale associated with a bounded observable and prove a martingale central limit theorem and a functional central limit theorem, showing that the limiting variance $\sigma_f^2$ is governed entirely by the refresh kernel $K^{\star}$ and the uniform measure $\pi,$ independently of the memory parameters $p_1,p_2, \dots, p_k.$ Finally, Section \ref{sec:7} computes $\sigma_f^2$ explicitly for eight natural observables on the hypercube.

\section{The coordinate-wise Elephant Random Walk}\label{sec:2}
We introduce the Coordinate-wise Elephant Random Walk (CERW), as an extension of the one-dimensional Elephant Random Walk (ERW) on $\mathbb Z$ \cite{Schutz2004}. The ERW is a random walk that incorporates the memory of its history. It moves like elephants, who possess remarkably long-term spatial memory.   
Let $(\Omega,\F,\PP)$ be a probability space where all the random variables below are defined. Fix $k \in \N,$ and let $Q_k:=\{-1,1\}^k$ denote the hypercube of dimension $k$. The state $S_n$ of the walk at time $n$ is 
\begin{equation*}
    S_n=(X_n^{(1)}, X_n^{(2)}, \dots, X_n^{(k)}) \in Q_k.
\end{equation*}
For $x \in Q_k$, we write $x^{i\to s}$ in section \ref{sec:7} for the vector obtained from $x$ by replacing its $i-$th coordinate by $s \in \{-1,1\}$ and leaving all other coordinates unchanged.  
Let $p_1,p_2,\dots,p_k\in [0,1],$ and define $\alpha_i = 2p_i-1$ for $1\leq i \leq k.$ 
Let $(I_n)_{n\geq 1}$ be a sequence of independent uniform random variables on $\{1, \ldots , k\}$, independent of the initial state $S_0$, that is, 
\begin{equation} \label{eq:In}
    I_n\sim \Unif\{1,\dots,k\}, \quad n \ge 1. 
\end{equation}
The variable $I_n$ indicates which coordinate is selected for update at time $n$.
For each coordinate $i\in \{1,\dots,k\}$, define its update times $(\tau_r^{(i)})_{r \ge 1}$  recursively by 
\begin{equation}\label{eq:updt}
    \tau_0^{(i)}=0, \qquad \tau_r^{(i)}:=\inf\{n>\tau_{r-1}^{(i)}:I_n=i\}, \qquad r\geq 1.
\end{equation}
Thus, $\tau_r^{(i)}$ is the $r-$th time that coordinate $i$ is selected.
We also define the number $N^{(i)}_n$ of updates of coordinate $i$ up to time $n$
\begin{equation} \label{eq:Nupdates}
    N_n^{(i)}:=\max\{r \geq 0 : \tau_r^{(i)}\leq n \}, \quad n \ge 0 . 
\end{equation}
 The history $\mathcal H^{(i)}_n$ of coordinate $i$ up to time $n$ is the finite sequence 
\begin{equation*}
    \calH_n^{(i)}:=\left(X_{\tau_0^{(i)}}^{(i)},X_{\tau_1^{(i)}}^{(i)}, \dots , X_{\tau_{N_n^{(i)}}^{(i)}}^{(i)} \right).
\end{equation*}
In other words, $\calH_n^{(i)}$ records the successive values taken by coordinate $i$ at its own update times until global time $n$. We include the initial value $ X_0^{(i)}$, so the length of $\calH_n^{(i)}$ is $N_n^{(i)}+1.$
Let $\F_n:=\sigma(S_0,S_1,\dots,S_n,I_1,\dots,I_n)$, $n \ge 1$, $\mathcal F_0 :=\sigma (S_0) $ form the natural filtration. Given $\F_n,$ the transition of $S_n = (X_n^{(1)}, \ldots , X_n^{(k)})$  from time $n$ to time $n+1$ is defined as follows.
\begin{enumerate}
    \item Choose a random coordinate $I_{n+1}$ uniformly on $\{1, \ldots , k\}$ as in \eqref{eq:In}, 
    \item Sample $\varepsilon_{n+1}$ uniformly from the history $\calH_n^{(I_{n+1})}$ of the chosen coordinate $I_{n+1}$, 
    \item Update the selected coordinate according to 
    \begin{equation}
    \label{het_eq:update-rule}X_{n+1}^{(I_{n+1})}= \begin{cases}
            \,\,\,\,\;\;\varepsilon_{n+1} \quad &\text{with probability } \, p_{I_{n+1}}, \\ 
            {} -\varepsilon_{n+1} \quad &\text{with probability } \, 1-p_{I_{n+1}} \\ 
        \end{cases}
    \end{equation} 
    and keep all other coordinates unchanged: \begin{equation*}
        X_{n+1}^{(j)} =X_{n}^{(j)} \quad
        \text{for}\; j\neq I_{n+1} . 
    \end{equation*}
\end{enumerate}
The resulting process $(S_n)_{n\geq0}$ is called the {\it Coordinate-wise Elephant Random Walk} (CERW) on $Q_k.$

This construction is a coordinate-wise analogue of the one-dimensional elephant random walk. In the original ERW \cite{Schutz2004} on $\mathbb Z$, the next increment of the ERW is obtained by sampling a previous increment uniformly from the history and then using it with probability $p$ or reversing it with probability $1-p$. Here in the CERW, the same history $\mathcal H^{(i)}_n$ is applied separately to each coordinate $i$ of the hypercube $Q_k$. Thus, each coordinate has its own update history, while the global process $(S_n)_{n\ge 0}$ evolves in finite state space $Q_k$ by updating one coordinate at a time. Consequently, $(S_n)_{n\ge 0}$ is a nearest-neighbor walk on the hypercube with possible holding, but its transition rule depends on the coordinate-wise histories and it is therefore non-Markovian on $Q_k.$

The CERW $(S_n)_{n\ge 0}$ should not be confused with a collection of $k$ {\it independent} one-dimensional ERWs. Its coordinates  share a common clock: at each global step $n = 0, 1, 2, \ldots$, only one coordinate is uniformly selected and updated by \eqref{het_eq:update-rule}, so the number of updates per coordinate by time $n$ is random. This shared-clock coupling gives the CERW its genuine multivariate character, even though its memory is defined separately for each coordinate.
\section{One-Coordinate Dynamics}\label{sec:3}
\noindent We start our analysis of the CERW with a fixed coordinate $i.$ For notational simplicity, we define the embedded process and its filtration 
\begin{equation}\label{eq:Yr}
    Y^{(i)}_r:=X_{\tau_r^{(i)}}^{(i)},  \qquad \mathcal{G}^{(i)}_r:=\sigma(Y^{(i)}_0,\dots,Y^{(i)}_r), \quad r \ge 0  
\end{equation}
along its update times $\tau^{(i)}_r$, $r \ge 0$ in \eqref{eq:updt}. We also  write $x_r^{(i)}$ the proportion of embedded values equal to $1$ among indices $0, 1, \ldots , r$ by  
\begin{equation} \label{eq:xri}
    Z^{(i)}_r =\sum_{\ell=0}^{r} \1_{\{Y^{(i)}_{\ell}=1\}},\qquad x^{(i)}_r:=\frac{Z^{(i)}_r}{r+1} \qquad \text{and} \qquad  m^{(i)}_r:=2x^{(i)}_r-1, \quad r \ge 0,  
\end{equation}
where $m_r^{(i)}$ is the empirical mean of the embedded coordinate values $Y_0^{(i)},Y_1^{(i)}, \dots, Y_r^{(i)} $
\begin{equation} \label{eq:mri}
   m_r^{(i)}=\dfrac{1}{r+1}\sum_{\ell=0}^r Y_{\ell}^{(i)}, \quad r \ge 0 . 
\end{equation}
With the number $N^{(i)}_n$ of updates in \eqref{eq:Nupdates} for the global time $n$, we define the corresponding global time memory bias $\widehat{m}_n^{(i)}$ by
\begin{equation} \label{eq:mhat}
\widehat{m}_n^{(i)}:=m^{(i)}_{N_n^{(i)}}, \quad n \ge 0 . 
\end{equation}
We shall see that these quantities will appear naturally in our analysis of the CERW. 
\subsection{Embedded process and martingale difference}
\begin{lemma} \label{het_lem:embedded-mean}
    With \eqref{eq:mri} and $\alpha_i = 2 p_i -1$ the embedded process $(Y^{(i)}_r)_{r \ge 0}$ in \eqref{eq:Yr} satisfies 
    \begin{equation} \label{eq:lem3.1}
        \E\left[Y^{(i)}_{r+1}\mid \G^{(i)}_r\right]=\alpha_im^{(i)}_r, \quad r \ge 0 . 
    \end{equation}
\end{lemma}
\begin{proof}
    By the definition of the CERW,  conditional on $\G_r^{(i)}$, the next update of coordinate $i$ is obtained by uniform sampling from $Y_0^{(i)},Y_1^{(i)}, \dots, Y_r^{(i)}.$ Let $\varepsilon$ denote the sample value. Then
\begin{equation*}
\E[\varepsilon\mid \G_r^{(i)}]
=
\frac{1}{r+1}\sum_{\ell=0}^{r}Y_\ell^{(i)}
=
m_r^{(i)}.
\end{equation*}
Given $\varepsilon$, the new value is $\varepsilon$ with probability $p_i$ and $-\varepsilon$ with probability $1-p_i$. Hence
\begin{equation*}
\E[Y_{r+1}^{(i)}\mid \varepsilon,\G_r^{(i)}]
= p_i\varepsilon-(1-p_i)\varepsilon = (2p_i-1)\varepsilon = \alpha_i\varepsilon.
\end{equation*}
By the tower property of conditional expectations, taking the conditional expectation once again gives \eqref{eq:lem3.1} 
\begin{equation*}
\E[Y_{r+1}^{(i)}\mid \G_r^{(i)}]=\alpha_i\E[\varepsilon\mid \G_r^{(i)}]=\alpha_i m_r^{(i)}, \quad r \ge 0 .
\end{equation*}  
\end{proof}
\begin{lemma}\label{het_lem:embedded-recursion}
    Define a martingale difference sequence $\eta_{r+1}^{(i)} := \1_{\{Y_{r+1}^{(i)}=1\}}-\E[\1_{\{Y_{r+1}^{(i)}=1\}}\mid\G_r^{(i)}]$, $r \ge 0 $ with $\E[\eta^{(i)}_{r+1}\mid \G ^{(i)}_r]=0$ and $\E[(\eta ^{(i)}_{r+1})^2\mid \G^{(i)}_r] \leq {1}/{4}$.
    $(x_r^{(i)})_{r\ge 0}$ in \eqref{eq:xri} satisfies 
    \begin{equation} \label{eq:Lem3.2}
        x^{(i)}_{r+1}-\frac{1}{2}=\left(1-\frac{2(1-p_i)}{r+2}\right)\left(x^{(i)}_r-\frac{1}{2} \right)+\frac{\eta^{(i)}_{r+1}}{r+2}, \quad r \ge 0 . 
    \end{equation}
\end{lemma}
\begin{proof}
Since $\1_{\{Y_{r+1}^{(i)}=1\}}={(1+Y_{r+1}^{(i)})}/2$ and $m_r^{(i)}=2x_r^{(i)}-1$ by definition, Lemma \ref{het_lem:embedded-mean} gives
\begin{equation*}\E\left[\1_{\{Y_{r+1}^{(i)}=1\}}\mid \G_r^{(i)}\right]=\frac{1+\alpha_i m_r^{(i)}}2 = \frac{1}{2}+\alpha_i\left(x_r^{(i)}-\frac{1}{2}\right), \quad r \ge 0 . 
\end{equation*}
It follows from \eqref{eq:xri} that $ Z_{r+1}^{(i)}=Z_r^{(i)}+\1_{\{Y_{r+1}^{(i)}=1\}}$, 
and hence
$ x_{r+1}^{(i)} = \frac{r+1}{r+2}x_r^{(i)}+\frac{1}{r+2}\1_{\{Y_{r+1}^{(i)}=1\}}$, $r \ge 0$. Thus, substituting the decomposition
\begin{equation*}
  \1_{\{Y_{r+1}^{(i)}=1\}}
= \E\left[\1_{\{Y_{r+1}^{(i)}=1\}} \mid \G_r^{(i)}\right] + \eta_{r+1}^{(i)} =
\frac{1}{2}+\alpha_i\left(x_r^{(i)}-\frac{1}{2}\right)+\eta_{r+1}^{(i)}  
\end{equation*}
into the recursion for $x_{r+1}^{(i)}$ gives
\begin{equation*}
  x_{r+1}^{(i)}-\frac{1}{2} = \frac{r+1+\alpha_i}{r+2} \left(x_r^{(i)}-\frac{1}{2}\right) +\frac{\eta_{r+1}^{(i)}}{r+2}.  
\end{equation*}
Since $r+1+\alpha_i = r+1+2p_i-1 = r+2-2(1-p_i),$ we obtain
\begin{equation*}
 \frac{r+1+\alpha_i}{r+2} = 1-\frac{2(1-p_i)}{r+2}.   
\end{equation*}
This proves the recursion \eqref{eq:Lem3.2}. 
\end{proof}
\subsection{Almost sure vanishing biases}
\begin{proposition}\label{het_thm:almost-sure-vanish}
    Fix a coordinate $i$. If $p_i<1,$ the empirical mean $m^{(i)}_r$ in \eqref{eq:mri} of the embedded process almost surely converges to $0$, as $r \to \infty$. 
\end{proposition}
\begin{proof}
By Lemma \ref{het_lem:embedded-recursion} with $c_i:=2(1-p_i)>0$ we see that  $u_r^{(i)}:=x_r^{(i)}-\frac{1}{2}$ satisfies  
\begin{equation*}
 u_{r+1}^{(i)} = \left(1-\frac{c_i}{r+2}\right)u_r^{(i)} +\frac{\eta_{r+1}^{(i)}}{r+2}, \quad 
 r \ge 0 ,    
\end{equation*}
Thanks to the properties of the martingale difference $\eta^{(i)}_\cdot$ in Lemma \ref{het_lem:embedded-recursion},  taking the conditional expectations of the square of $u^{(i)}_{r+1}$, 
we get
\begin{equation}\label{eq:square-uri}
 \E[(u_{r+1}^{(i)})^2\mid \G_r^{(i)}]
\leq \left(1-\frac{c_i}{r+2}\right)^2(u_r^{(i)})^2+\frac{1}{4(r+2)^2}, \quad r \ge 0 .    
\end{equation}
Since $c_i>0$, this recursion implies that $u_r^{(i)}\to0$ almost surely, as $r \to \infty$ by the Robbins--Siegmund almost supermartingale theorem found in \cite{RobbinsSiegmund1971}. Therefore, as $r\to \infty$, 
$m_r^{(i)}=2u_r^{(i)}$ converges almost surely to $0$. 
\end{proof}
\begin{corollary}\label{het_cor:almost-sure-vanish}
Assume $\max_{1\le i \le k} p_i<1$. Then with probability one, the global time memory bias $\widehat m_n^{(i)} = m^{(i)}_{N_n^{(i)}}$ in \eqref{eq:mhat} converges to $0$, as $n \to \infty$ for every $i = 1, \ldots , k$. 
\end{corollary}

\begin{proof}
Fix $i\in\{1,\dots,k\}$ first. By Proposition 
\ref{het_thm:almost-sure-vanish}, it suffices to show that the number $N_n^{(i)}$ of updates of coordinate $i$ tends to
infinity, as $n \to \infty$. For $n\ge1$, let $A_n^{(i)}:=\{I_n=i\}$. Since the events $(A_n^{(i)})_{n\ge1}$ are independent and $\PP(A_n^{(i)})=\frac{1}{k}$, 
\begin{equation*}
 \sum_{n=1}^{\infty}\PP(A_n^{(i)}) = \sum_{n=1}^{\infty}\frac1k = \infty.   
\end{equation*}

By the second Borel--Cantelli lemma, $\PP(A_n^{(i)}\ \text{infinitely often})=1$.
Therefore, coordinate $i$ is selected infinitely many times almost surely,
which is equivalent to $N_n^{(i)}\to\infty$ almost surely.
Recall that $\widehat m_n^{(i)}=m_{N_n^{(i)}}^{(i)}.$
On the event where both
$m_r^{(i)}\to0$ and
$N_n^{(i)}\to\infty,$
the sequence $\widehat m_n^{(i)}$ is a subsequence, with repetitions, of
$(m_r^{(i)})_{r\ge0}$ whose index tends to infinity. Consequently, $\widehat m_n^{(i)} \to0$, as $n \to \infty$ almost surely. 
Since there are only
$k$ coordinates, intersecting the corresponding 
probability-one events gives $ \widehat m_n^{(i)}\to0$ 
for $i=1,\dots,k$.
\end{proof}
\subsection{Rate of decay of biases}
\begin{proposition}
    Assume $p_i\in(0,1).$ There exists $C_{p_i}<\infty$ such that for all $r \geq 1,$
    \begin{equation*}
        \E[(m_r^{(i)})^2] \leq C_{p_i} \begin{cases}
            r^{-1}, & \text{\rm if }\, 0<p_i<\tfrac34,\\
            \dfrac{\log r}{r}, &\text{\rm if } \, p_i=\tfrac34,\\
            r^{-4(1-p_i)}, &\text{\rm if } \, \tfrac34<p_i<1.
        \end{cases}
    \end{equation*}
\end{proposition}

\begin{proof}
Let $a_r:=\E[(u_r^{(i)})^2]$. As in the proof of Proposition \ref{het_thm:almost-sure-vanish}, taking expectations of \eqref{eq:square-uri}, we obtain $a_{r+1} \leq \left(1-\frac{c_i}{r+2}\right)^2 a_r + \frac{1}{4(r+2)^2}$, $r \ge 0 $. 
Iterating this inequality for $r$ times gives
\begin{equation*}
  a_r \leq C_{p_i}r^{-2c_i} + C_{p_i}r^{-2c_i}\sum_{j=2}^{r} j^{2c_i-2}  , 
\end{equation*}
where $C_{p_i}$ is a positive constant which depends only on $p_i$ through $c_i = 2(1-p_i)$. The behavior of the last sum depends on whether $2c_i-2<-1, 2c_i-2=-1$ or $2c_i-2>-1$. Since $c_i=2(1-p_i)$, the critical case is $2c_i - 2 = -1 $ or equivalently $p_i=\frac{3}{4}$. Therefore, we have 
\begin{equation*}
 a_r\leq C_{p_i} 
 \begin{cases}
r^{-1}, & 0<p_i<\frac34,\\
\dfrac{\log r}{r}, & p_i=\frac34,\\
r^{-4(1-p_i)}, & \frac34<p_i<1. \end{cases}   
\end{equation*}
Finally, we conclude because $m_r^{(i)}=2u_r^{(i)}$, so $\E[(m_r^{(i)})^2]=4a_r$,
and the constant $4$ is absorbed into $C_{p_i}$.
\end{proof}
The phase transition at $p_i=\frac{3}{4}$ mirrors the classical trichotomy of the one-dimensional ERW \cite{Bercu2018, Schutz2004}, where the same threshold separates diffusive, critical and superdiffusive behavior of the walk's displacement: even for $p_i>\frac{3}{4},$ the mean squared bias still decays to zero, but at a slower polynomial rate. 

\section{Comparison with the refresh kernel}\label{sec:4}
Let $K^\star$ be the transition kernel of the following {\it memoryless refresh} chain on the hypercube $Q_k.$ From state $x \in Q_k$, the chain chooses a coordinate $I$ uniformly from $\{1,2,\dots,k\}$ with probability $1/k$ and replaces $x^{I}$ by an independent symmetric Rademacher variable. That is, 
\begin{equation} \label{eq:refreshK}
    K^\star(x,\cdot)=\dfrac{1}{k}\sum_{i=1}^kK^{\star}_i(x,\cdot), \quad x \in Q_k, 
\end{equation}
where $K^{\star}_i$ leaves all coordinates except $i$ unchanged and sets coordinate $i$ equal to $\pm1$ with probability ${1}/{2}$ each. 
Note that the uniform probability measure $\pi$ on $Q_k$ defined by  $\pi(x):={1}/{2^k}$,  $x\in Q_k$,  is invariant for $K^{\star}$. We call $K^\star$ the {\it refresh kernel}. Some examples are discussed later in section \ref{sec:7} (see \eqref{eq:Kstarf}). 

At time $n,$ conditional on the past, the CERW has a random transition kernel $K_n$ on $Q_k.$ More precisely, if the current state is $x=S_n,$ then the coordinate $i$ is selected with probability ${1}/{k},$ all other coordinates remain fixed, and coordinate $i$ is set to $1$ with probability $(1+\alpha_i\widehat{m}_n^{(i)})/2$. See \eqref{eq:lem3.1} of Lemma \ref{het_lem:embedded-mean}. 
Equivalently, for every bounded function $g:Q_k\to \R,$
\begin{equation}\label{het_eq:one-step-kernel}
    \E[g(S_{n+1})\mid \F_n]=(K_ng)(S_n), \quad n \ge 0.
\end{equation}
We shall compare two kernels $K^\star$ in \eqref{eq:refreshK} and $K_n$ in \eqref{het_eq:one-step-kernel}. 
\subsection{Total variation perturbation}
\begin{lemma}\label{het_lem:tv-bound}
    The total variation $\delta_n:=\sup_{x\in Q_k} ||K_n(x,\cdot)-K^{\star}(x,\cdot)||_{\TV}$, $n \ge 0$ satisfies  
    \begin{equation}
        \delta_n\leq \frac{1}{2}\max_{1\leq i \leq k} |\widehat{m}_n^{(i)}|.
    \end{equation}
     In particular, if $p_i<1$ for every $i$, then $\delta_n\to 0$ almost surely and in $L^1$, as $n \to \infty$. 
\end{lemma}
\begin{proof}
    For a fixed selected coordinate $i,$ the only difference between $K_n$ and $K^{\star}$ is the Bernoulli  law used to refresh coordinate $i.$ Under $K_n,$ the probability of $+1$ is $(1+\alpha_i\widehat{m}_n^{(i)})/2$, while under $K^{\star}$ is $1/2.$ Therefore, the total variation distance between these two laws is ${|\alpha_i\widehat{m}_n^{(i)}|}/{2}.$
    Averaging over the uniformly selected coordinate gives 
    \begin{equation*}
        \delta_n\leq \dfrac{1}{2k}\sum_{i=1}^k|\alpha_i\widehat{m}^{(i)}_n|\leq \frac{1}{2}\max_{1\leq i \leq k} |\widehat{m}_n^{(i)}|.
    \end{equation*}
    The almost sure convergence follows from Corollary \ref{het_cor:almost-sure-vanish}. Since $0\leq \delta_n \leq 1,$ bounded convergence also gives $\E[\delta_n]\to0.$
\end{proof}
\begin{lemma}\label{het_lem:finite-horizon}
    For every bounded $g:Q_k\to\R,$ every $n\geq0,$ and every integer $T\geq1,$
    \begin{equation}\label{het_eq:finite_horizon_comparison}
        \left|\E[g(S_{n+T})|\F_n]-((K^{\star})^Tg)(S_n) \right| \leq 2\| g\|_{\infty}\E\left[ \sum_{r=0}^{T-1}\delta_{n+r} \mid \F_n\right] \text{ a.s.}
    \end{equation}
    Consequently, we have 
    \begin{equation}\label{het_eq:finite_horizon_L1}
        \E\left[ |\E\left[g(S_{n+T})\mid \F_n \right] - ((K^{\star})^Tg)(S_n) | \right] \longrightarrow 0, \quad \text{as } n\to\infty.
    \end{equation}
\end{lemma}
\begin{proof}
    For $n\geq 0, T\geq 1$ and bounded $g:Q_k\to\R,$ set the difference 
    \begin{equation*}
        R_{n,T}:=\E\left[g(S_{n+T})\mid \F_n \right]- ((K^{\star})^Tg)(S_n).
    \end{equation*}
    We shall prove \eqref{het_eq:finite_horizon_comparison} by mathematical induction in $T$ for every starting time $n$ and every bounded function $g$. 
    
    \medskip 
    \noindent\textbf{Base case.} Let $T=1,$ by \eqref{het_eq:one-step-kernel}, we have 
        $R_{n,1}= (K_ng)(S_n)-(K^{\star}g)(S_n)$.
    Recall that for probability measures $\mu,\nu$ and bounded function $h$, the inequality $|\mu(h)-\nu(h)|\leq 2\|h\|_{\infty}\|\mu - \nu\|_{TV}$ holds, where we use the notation $\mu(h) := \int h \mathrm d \mu $.  Applying this inequality with the kernels $\mu=K_n(S_n,\cdot)$ and $\nu=K^{\star}(S_n,\cdot)$ gives
    \begin{equation*}
        |R_{n,1}(g)|\leq 2\|g\|_{\infty} \, \|K_n(S_n,\cdot)-K^{\star}(S_n,\cdot) \|_{TV} \leq 2\|g\|_{\infty}\, \delta_n.
    \end{equation*}
    Since $\delta_n$ is $\F_n$-measurable, $\delta_n=\E[\delta_n\mid \F_n]$. Thus, the desired bound follows for $T=1.$
    \medskip
    
    \noindent\textbf{Inductive hypothesis.}
    Assume that the estimate \eqref{het_eq:finite_horizon_comparison} holds for some $T\geq1, $ for every starting time and for every bounded test function. 
    \medskip 
    
    \noindent\textbf{Inductive step.} We will prove  \eqref{het_eq:finite_horizon_comparison} for $T+1.$ By the tower property, $\E[g(S_{n+T+1})\mid \F_n]= \E[ \E[g(S_{n+T+1})\mid \F_{n+1} ] \mid \F_n]$. By definition of $R_{n+1,T}(g),$ $\E[g(S_{n+T+1})\mid \F_{n+1}]=((K^{\star})^Tg)(S_{n+1})+R_{n+1,T}(g).$ Set $h_T:=(K^{\star})^Tg\,$, which is a deterministic bounded function. Therefore, the one step identity \eqref{het_eq:one-step-kernel} gives us, $\E[h_T(S_{n+1})\mid \F_n]=(K_nh_T)(S_n).$ Since $K^{\star}h_T=(K^{\star})^{T+1}g,$ we obtain 
    \begin{equation}\label{het_eq:induction-error-decomposition}
        R_{n,T+1}(g)=\left[(K_nh_T)(S_n)-(K^\star h_T)(S_n)\right]+\E[R_{n+1,T}(g)\mid\F_n].
    \end{equation}
    We estimate the two terms in \eqref{het_eq:induction-error-decomposition}. Since $K^\star$ is a Markov kernel, $\|h_T \|_{\infty}=\|(K^{\star})^Tg\|_{\infty}\leq \| g\|_{\infty}.$ Hence,
    \begin{equation*}
        |(K_nh_T)(S_n)-(K^{\star}h_T)(S_n) | \leq 2\| g\|_{\infty}\delta_n.
    \end{equation*}
    For the second term, Jensen's inequality and the induction hypothesis give
    \begin{align*}
        |\E[R_{n+1,T}(g)\mid \F_n]| &\leq \E[|R_{n+1,T}(g)|\mid \F_n]\\
        & \leq 2\| g\|_{\infty}\E\left[\E\left[\sum_{r=0}^{T-1}\delta_{n+1+r}\mid \F_{n+1}\right] \mid \F_n\right]\\
        &= 2\| g\|_{\infty}\E\left[\sum_{r=0}^{T-1}\delta_{n+1+r}\mid \F_n \right],
    \end{align*}
    where the last equality follows from the tower property. Combining the two estimates and again using that $\delta_n$ is $\F_n$-measurable, we get
    \begin{equation*}
        |R_{n,T+1}(g)|\leq 2\| g\|_{\infty}\E\left[ \sum_{r=0}^T\delta_{n+r}\mid \F_n\right],
    \end{equation*}
    which is the desired estimate for the  horizon $T+1.$ Thus, the induction is complete. 
    Taking expectations in \eqref{het_eq:finite_horizon_comparison} gives
    \begin{equation*}
        \E[|R_{n,T}(g)|]\leq 2\|g\|_{\infty}\sum_{r=0}^{T-1}\E[\delta_{n+r}].
    \end{equation*}
    For a fixed $T$, the right-hand side tends to zero because $\E[\delta_m]\to0.$ This proves \eqref{het_eq:finite_horizon_L1}.
\end{proof}

\begin{lemma}
    Let $t_0:=k\lceil \ln(2k)\rceil.$ Then, for every $x \in Q_k$ and every $A \subseteq Q_k,$
    \begin{equation}\label{het_eq:minorization}
        (K^{\star})^{t_0}(x,A) \geq \frac{1}{2}\pi(A), \qquad x\in Q_k ,\quad A\subseteq Q_k.
    \end{equation}
    Consequently, for every probability measure $\mu$ on $Q_k,$ 
    \begin{equation}\label{het_eq:uniform-ergodicity-TV}
        ||\mu (K^{\star})^{mt_0}- \pi||_{\TV} \leq 2^{-m}, \qquad m\geq0.
    \end{equation}
    In particular, for every bounded $g:Q_k\to\R,$
    \begin{equation}\label{het_eq:uniform-ergodicity-functions}
        \sup_{x\in Q_k} \left|((K^{\star})^{mt_0}g)(x)-\pi(g) \right| \leq ||g||_{\infty}2^{-(m-1)}, \qquad m\geq 0.
    \end{equation}
\end{lemma}
\begin{proof}
    For every $t\geq1,$ let $E_t$ be the event that every coordinate is selected at least once during the first $t$ steps of the refresh chain. On $E_t,$ each coordinate has been refreshed at least once, and the last refresh of each coordinate determines its values at time $t.$ These last refreshed values are independent Rademacher random variables. Therefore, conditional on $E_t,$ the state at time $t$ has law $\pi,$ independently of the initial state. Hence, for every $x\in Q_k$ and every $A\subseteq Q_k,$
    \begin{equation*}
        (K^{\star})^t(x,A)=\PP_x(S_t
        \in A) \geq \PP_x(S_t \in A, E_t) = \PP(E_t)\pi(A).
    \end{equation*}
    
    It remains to bound $\PP(E_t)$ from below. By the union bound, 
    \begin{equation*}
        \PP(E_t^c) \leq \sum_{i=1}^k \PP(\text{coordinate }i \text{ is never selected during the first }t \text{ steps}) = k\left(1-\dfrac{1}{k} \right)^t.
    \end{equation*}
    Using $1-u\leq e^{-u},$ we get $\PP(E^c_t)\leq k\exp(-{t}/{k})$, $t \ge 0$. With $t=t_0=k\lceil \ln(2k) \rceil,$ this yields $\PP(E^c_{t_0})\leq k\exp (-t_0/{k}) \leq k\exp(-\ln(2k))={1}/{2}$. Thus, $\PP(E_{t_0})\geq 1/2,$ and therefore $(K^{\star})^{t_0}(x,A)\geq \pi(A)/2$,
    which proves \eqref{het_eq:minorization}.
    
    We now derive the contraction estimate. By the standard minorization argument (see \cite{Rosenthal1995ConvergenceRates}), it follows that for every probability measure $\mu$ on $Q_k,$ 
    \begin{equation*} 
        ||\mu(K^{\star})^{mt_0}-\pi||_{TV} \leq \Big(\frac{1}{2}\Big)^m,\qquad m\geq 0,
    \end{equation*}
    which proves \eqref{het_eq:uniform-ergodicity-TV}. Finally, taking $\mu=\delta_x,$ we have 
    \begin{equation*}
        |((K^{\star})^{mt_0}g)(x)-\pi(g)|\leq 2||g||_{\infty}||\delta_x(K^{\star})^{mt_0}-\pi||_{\TV}.
    \end{equation*}
    Using \eqref{het_eq:uniform-ergodicity-TV} and taking the supremum over $x\in Q_k$ yields
    \begin{equation*}
        \sup_{x\in Q_k} \left|((K^{\star})^{mt_0}g)(x)-\pi(g) \right| \leq 2||g||_{\infty}2^{-m}, \qquad m\geq 0.
    \end{equation*}
\end{proof}

\section{Asymptotic stationarity and weak laws}\label{sec:5}
We now prove a weak law  of large numbers for bounded observables below in Theorem \ref{het_thm:WLLN}. The limiting measure is the uniform probability measure $\pi$ on the unit hypercube $Q_k$ (see section \ref{sec:4}). Note that $\pi$ is invariant for the refresh kernel $K^\star$ in \eqref{eq:refreshK}. The proof is organized in two steps. First, we prove a weak law along the block skeleton $(S_{bt_0})_{b\geq0}.$ Then we recover the full empirical average by decomposing each block into a martingale fluctuation and a predictable conditional mean. 

\begin{lemma}[Skeleton Weak Law] \label{het_lem:WLLN-skeleton}
    Assume that $\max_{1\le i \le k} p_i<1$. Let $t_0$ be the block length of \eqref{het_eq:minorization}. Then, for every bounded function $g:Q_k\to\R$, we have  
    \begin{equation} \label{eq:lem5.1}
        \dfrac{1}{M}\sum_{b=0}^{M-1}g(S_{bt_0}) \xrightarrow[M\to\infty]{\PP} \pi(g).
    \end{equation}
\end{lemma}
\begin{proof}
    For fixed integers $b\ge m$ and a  bounded  function $g$, we define 
    \begin{equation*}
        P_b^{(m)} :=\E[g(S_{bt_0})\mid\F_{(b-m)t_0}], \qquad W_b^{(m)}:=g(S_{bt_0})-P_b^{(m)}.
    \end{equation*}
 Then $\E[W_b^{(m)}\mid \F_{(b-m)t_0}]=0$, and $ |W_b^{(m)}|\leq 2\|g\|_{\infty}$. We write $g(S_{bt_0})=W_b^{(m)}+P_b^{(m)}$. 
 
 First, we control the first term $W_b^{(m)}$. Taking a large integer $c\geq b+m$ so that $bt_0\leq (c-m)t_0$, we see  
      \begin{equation*}
      \E\left[W_b^{(m)}W_c^{(m)}\right] = \E\left[W_b^{(m)}\E\left[W_c^{(m)}|\F_{(c-m)t_0}\right]\right]=0
    \end{equation*}
    since the random variable $W_b^{(m)}$ is $\F_{(c-m)t_0}$-measurable. 
    By symmetry, the covariance also vanishes when $b\geq c+m.$ Thus, only the pairs $(b,c)$ with $|b-c|<m$ can contribute to the sum. There are at most $M(2m-1)$ such ordered pairs. Therefore, we have an upper bound 
    \begin{equation*}
        \E\left[\left(\dfrac{1}{M}\sum_{b=m}^{M-1}W_b^{(m)}\right)^2\right]\leq \dfrac{4(2m-1)\|g\|_{\infty}^2}{M}.
    \end{equation*}
    For a fixed $m$, as $M\to \infty$, we obtain 
    \begin{equation}\label{het_eq:skeleton-fluctuation}
        \dfrac{1}{M}\sum_{b=m}^{M-1}W_b^{(m)}\longrightarrow 0 \qquad \text{in } L^2.
    \end{equation}
    
 Next, we control the second term $P_b^{(m)}$, using lemma \ref{het_lem:finite-horizon} with $n=(b-m)t_0$ and $T=mt_0$ to obtain 
    \begin{equation*}
        \left|P_b^{(m)} - ((K^{\star})^{mt_0}g)(S_{(b-m)t_0}) \right|\leq 2\|g\|_{\infty}\E\left[\sum_{s=(b-m)t_0}^{bt_0-1}\delta_s \mid \F_{(b-m)t_0} \right]
    \end{equation*}
    Averaging, taking absolute values, and then taking expectations, we obtain 
    \begin{equation} \label{eq:ExpAverage}
        \E\left[\left|\dfrac{1}{M}\sum_{b=m}^{M-1}\left(P_b^{(m)}- ((K^{\star})^{mt_0}g)(S_{(b-m)t_0})\right) \right| \right] \leq 2\dfrac{\|g\|_{\infty}}{M}\sum_{b=m}^{M-1}\sum_{s=(b-m)t_0}^{bt_0-1}\E[\delta_s].
    \end{equation}
    Each index $s$  belongs to at most $m$ of the intervals $[(b-m)t_0,bt_0-1]$ as they have length $mt_0$ and their left endpoints are separated by $t_0.$ Thus, the right-hand side of \eqref{eq:ExpAverage} is bounded by ${2m\|g\|_{\infty}}\sum_{s=0}^{Mt_0-1}\E[\delta_s]/M$. 
    Since $\E[\delta_s]\to0$ by Lemma \ref{het_lem:tv-bound}, Ces\`aro's theorem gives $\sum_{s=0}^{Mt_0-1}\E[\delta_s]/M\to0$ as $M\to \infty$. Therefore, for fixed $m$, we have $L^1$ convergence: 
    \begin{equation}\label{het_eq:skeleton-perturbation}
        \dfrac{1}{M}\sum_{b=m}^{M-1}\left(P_b^{(m)}- ((K^{\star})^{mt_0}g)(S_{(b-m)t_0})\right) \xrightarrow[M\to \infty]{L^1} 0 
    \end{equation}
    
    By the uniform mixing estimate \eqref{het_eq:uniform-ergodicity-functions}, 
    \begin{equation*}
        \left| \dfrac{1}{M}\sum_{b=m}^{M-1} ((K^{\star})^{mt_0}g)(S_{(b-m)t_0}) - \dfrac{M-m}{M}\pi(g)\right| \leq 2\|g\|_{\infty}2^{-m}.
    \end{equation*}
    The sum of the first $m$ terms satisfies 
    \begin{equation*}
        \left| \dfrac{1}{M}\sum_{b=0}^{m-1} (g(S_{bt_0}) - \pi(g))\right| \leq \frac{2m\|g \|_{\infty}}{M} \xrightarrow[M\to \infty]{} 0.
    \end{equation*}
    Combining this initial term bound with \eqref{het_eq:skeleton-fluctuation}, \eqref{het_eq:skeleton-perturbation} and the mixing estimate above, we get, for every fixed $m$ and every $\varepsilon>0,$
    \begin{equation*}
        \limsup_{M\to\infty}\PP\left(\left| \dfrac{1}{M}\sum_{b=0}^{M-1} g(S_{bt_0})-\pi(g)\right|>2\|g\|_{\infty}2^{-m}+\varepsilon \right)=0.
    \end{equation*}
    Since $m$ is arbitrary and $2\|g\|_{\infty}2^{-m}\to0$, as $m\to \infty$, the skeleton  weak law \eqref{eq:lem5.1} follows. 
    \end{proof}

\begin{lemma}\label{het_lem:block-replacement}
    Let $f:Q_k\to\R$ be bounded. Write $n=Mt_0+r,$ with $0\leq r <t_0.$ Define
    \begin{equation}\label{eq:lem5.2}
        A_n:=\frac{1}{n}\sum_{t=0}^{n-1}f(S_t) \quad \text{and} \quad \widetilde{A_n}:=\frac{1}{Mt_0}\sum_{b=0}^{M-1}\sum_{j=0}^{t_0-1} f(S_{bt_0+j}) , \quad n \ge 1 .  
    \end{equation}
    Then we have an estimate for the difference 
    \begin{equation*}
        |A_n-\widetilde{A_n}|\leq \dfrac{2t_0||f||_{\infty}}{n}, \quad n \ge 1 .  
    \end{equation*}
    In particular, $A_n-\widetilde{A_n}\to0$, as $n \to \infty$. 
\end{lemma}
\begin{proof}
    With $C_M:=\sum_{t=0}^{Mt_0-1}f(S_t)$ and   $R_n:=\sum_{t=Mt_0}^{n-1}f(S_t)$, the difference $A_n - \widetilde{A_n}$ can be rewritten as
    \begin{equation*}
        A_n-\widetilde{A_n}=\left(\dfrac{1}{n}-\dfrac{1}{Mt_0}\right)C_M+\dfrac{1}{n}R_n. 
    \end{equation*}
         Since $|C_M|\leq Mt_0\|f\|_{\infty}, |R_n|\leq r\|f\|_{\infty}$ and $r<t_0,$ we have the estimate 
    \begin{equation*}
        |A_n-\widetilde{A_n}|\leq \dfrac{r}{nMt_0}Mt_0\|f\|_{\infty}+\dfrac{r}{n}\|f\|_\infty \leq \dfrac{2t_0\|f\|_{\infty}}{n}\to 0, \quad \text{as } n\to\infty.
    \end{equation*}
\end{proof}
\begin{theorem}[Weak Law of Large Numbers] \label{het_thm:WLLN}
    Assume $\max_{1\le i \le k } p_i<1$. Then, for every bounded function $f:Q_k\to\R,$ we have the weak law of large numbers 
    \begin{equation*}
        \frac{1}{n}\sum_{t=0}^{n-1}f(S_t)\xrightarrow[n \to \infty]{\PP}\pi(f) . 
    \end{equation*}
\end{theorem}
\begin{proof}
    Write $n=Mt_0+r,$ where $t_0$ is the block size chosen as in Lemma \ref{het_lem:WLLN-skeleton} and $r$ is the remainder with $0\leq r < t_0$.  By lemma \ref{het_lem:block-replacement}, it is enough to prove that $ \widetilde{A_n}$ in \eqref{eq:lem5.2} 
    converges in probability to $\pi(f).$ For $j=0,\dots,t_0-1,$ we  define 
    \begin{equation} \label{eq:hx}
        h(x):=\frac{1}{t_0}\sum_{j=0}^{t_0-1}((K^{\star})^jf)(x),
    \end{equation}
    with the convention that $(K^{\star})^0f=f$. Note that since $\pi$ is invariant for $K^{\star},$
            $\pi(h)=\sum_{j=0}^{t_0-1}\pi((K^\star)^j f)/{t_0}=\pi(f)$.
            
    For each block $b = 0, 1, 2, \ldots ,$ we set the block sum $B_b :=\sum_{j=0}^{t_0-1}f(S_{bt_0+j})$, its conditional expectation $\bar{B}_b:=\E[B_b|\mathcal{F}_{bt_0}] = \sum_{j=0}^{t_0-1} \E [ f(S_{bt_0+j}) \vert \mathcal F_{bt_0}] $ and the deviation  $U_b:= B_b-\bar{B}_b$. 
    Let $\G_b:=\F_{bt_0}$ for $b = 0, 1,  \ldots $. Then $U_b$ is a $\G_{b+1}$-measurable random variable with $\E[U_b|\G_b]=0.$ Thus, $(U_b,\G_{b+1})_{b\geq0}$ is a martingale difference sequence. If $b<c,$ then $U_b$ is $\G_c$-measurable and hence $\E[U_bU_c]=\E[U_b\E[U_c\mid \G_c]]=0.$ Since $|U_b|\leq 2t_0\|f\|_{\infty},$ we have the upper bound 
     \begin{equation*}
        \E\left[\left(\frac{1}{Mt_0}\sum_{b=0}^{M-1}U_b\right)^2\right]=\dfrac{1}{M^2t_0^2}\sum_{b=0}^{M-1}\E[U_b^2]\leq\frac{4\|f\|_\infty^2}{M}\xrightarrow[M\to \infty]{} 0.
    \end{equation*}
    Thus, the average of the block sum $B_\cdot$ is approximated by the average of $\bar{B}_\cdot$ in $L^2$:  
    \begin{equation}\label{het_eq:block-fluctuation}
        \frac{1}{Mt_0}\sum_{b=0}^{M-1}(B_b-\bar{B}_b)\xrightarrow[M\to\infty]{L^2}0.
    \end{equation}
    
    It remains to compare the predictable block averages $\bar{B}_\cdot/t_0$ with  $h(S_{bt_0})$ from \eqref{eq:hx}. In the expression of $ \bar{B}_b$, the summand $ \mathbb E [ f(S_{b t_0 + j}) \vert \mathcal F_{b t_0}] $ becomes  $\E[f(S_{bt_0})\mid \F_{bt_0}]=f(S_{bt_0})=((K^{\star})^0f)(S_{bt_0})$ for $j = 0$; It becomes 
    \begin{equation*}
        |\E[f(S_{bt_0+j})\mid \F_{bt_0}]- ((K^{\star})^jf)(S_{bt_0})| \leq 2\|f\|_{\infty}\E\left[ \sum_{i=0}^{j-1} \delta_{bt_0+i}\mid \F_{bt_0}\right] , 
    \end{equation*}
     thanks to Lemma \ref{het_lem:finite-horizon}  with $n=bt_0$ and $T=j$ for $1 \leq j \leq t_0-1$. 
     Averaging over $j=0, 1,\dots,t_0-1$ and dividing by $t_0,$ we obtain
    \begin{equation}\label{het_eq:block-predictable-bound}
    \left|\dfrac{\bar{B}_b}{t_0}-h(S_{bt_0})\right|\leq 2\|f\|_{\infty}\E\left[ \sum_{i=0}^{t_0-1}\delta_{bt_0+i}\mid \F_{bt_0}\right].
    \end{equation}
    Moreover, averaging \eqref{het_eq:block-predictable-bound} over $b$ and taking expectations yields 
    \begin{equation*}
        \E\left[\left|\frac{1}{Mt_0}\sum_{b=0}^{M-1}\bar{B}_b-\frac{1}{M}\sum_{b=0}^{M-1}h(S_{bt_0}) \right|\right] \leq \frac{2||f||_{\infty}}{M}\sum_{s=0}^{Mt_0-1}\E[\delta_{s}].
    \end{equation*}
    Since $\E[\delta_s]\to0$ by Lemma \ref{het_lem:tv-bound}, the upper bound in the right-hand side converges to zero by Ces\`aro's theorem, and therefore,
    \begin{equation}\label{het_eq:block-perturbation}
        \frac{1}{Mt_0}\sum_{b=0}^{M-1}\bar{B}_b-\frac{1}{M}\sum_{b=0}^{M-1}h(S_{bt_0}) \xrightarrow[M\to \infty]{L^1} 0 . 
    \end{equation}
    Combining \eqref{het_eq:block-fluctuation} and \eqref{het_eq:block-perturbation} gives $\widetilde{A}_n - {M}^{-1}\sum_{b=0}^{M-1}h(S_{bt_0}) \xrightarrow[n\to\infty]{\PP}0$.

    Thus, by Lemma \ref{het_lem:WLLN-skeleton} ${M}^{-1}\sum_{b=0}^{M-1}h(S_{bt_0}) \xrightarrow[M\to\infty]{\PP}\pi(h)=\pi(f)$.
    Moreover, $\widetilde{A}_n  \xrightarrow[n\to\infty]{\PP}\pi(f).$ Finally, we conclude by Lemma \ref{het_lem:block-replacement} that 
    \begin{equation*}
        A_n  = \frac{1}{n} \sum_{t=0}^{n-1} f(S_t) \xrightarrow[n\to\infty]{\PP}\pi(f).
    \end{equation*}
\end{proof}

\section{Functional CLT for Bounded Observables}\label{sec:6}
In this section we shall prove the functional central limit theorem for the CERW $S_\cdot$ on the hypercube $Q_k$. Fix a bounded function $f: Q_k \to \R$ and define $M_0^{(f)}  = 0$, 
\begin{equation}
\begin{split}
    M_n^{(f)}:&= f(S_n)-f(S_0)-\sum_{t=0}^{n-1} \E \left[ f(S_{t+1})-f(S_t) 
    \mid \F_{t} \right] 
    = \sum_{t=0}^{n-1} D_{t+1}^{(f)} 
    \end{split}
\end{equation}
for $n \ge 1$, where $D_{t+1}^{(f)}:= f(S_{t+1})-f(S_t)- \E \left[ f(S_{t+1})-f(S_t) \mid \F_{t} \right]$, $t \ge 0 $. 
\begin{proposition}\label{het_prop:doob_martingale}
    $(M_n^{(f)}, \F_n)_{n\ge 0}$ is a square integrable martingale with predictable quadratic variation 
    $\langle M^{(f)}\rangle_n = \sum_{t=0}^{n-1} \Var(f(S_{t+1})\mid \F_t)$, $n \ge 1$ 
and  bounded martingale difference sequence $D^{(f)}_\cdot$ with $|D_{t+1}^{(f)}| \leq 4 ||f||_{\infty}$, $t \ge 0$. 
\end{proposition}
\begin{proof}
Since $f$ is bounded, $D_{t+1}^{(f)}$ is square integrable. By definition $\E[D_{t+1}^{(f)}|\F_t]=0$. Thus, $M_n^{(f)}=\sum_{t=0}^{n-1}D_{t+1}^{(f)}$ is a square integrable martingale with respect to filtration  $(\F_n)$. 
    $|f(S_{t+1})-f(S_t)|\leq 2\|f\|_{\infty}$  and  $|\E[f(S_{t+1})-f(S_t)|\F_t] |\leq 2\|f\|_{\infty}$ imply  $|D_{t+1}^{(f)}| \leq 4 \|f\|_{\infty}$. 
\end{proof}
We shall see that the conditional variance $\text{Var}(f(S_{t+1}) \vert \mathcal F_t) $ should  asymptotically behave as the corresponding variance of the Markov transition under the refresh kernel $K^{\star}$, since the memory biases $\widehat{m}^{(i)}_\cdot$ vanish as in Corollary \ref{het_cor:almost-sure-vanish}.  For $x \in Q_k,$ we define the variance 
\begin{equation} \label{eq:vfx}
    v_f(x):=\Var_{K^{\star}(x,\cdot)}(f) = K^{\star}f^2(x)-(K^{\star}f(x))^2.
\end{equation}

\begin{lemma}\label{het_lem:variance-comparison}
    There exists $C_f<\infty$ such that
    \begin{equation*}
        \left |\Var(f(S_{t+1}) \mid \F_t)-v_f(S_t) \right| \leq C_f \max_{1 \leq i \leq k} |\widehat{m}_t^{(i)}|, \quad t \ge 0 , 
    \end{equation*}
    where $\widehat{m}_t^{(i)}$ is the memory bias defined in \eqref{eq:mhat}. In particular, if $\max_{1\le i \le k} p_i<1$, then the conditional variance is approximated by $v_f$ in \eqref{eq:vfx}: 
    \begin{equation*}
        \left |\Var(f(S_{t+1}) \mid \F_t)-v_f(S_t) \right| \xrightarrow[t\to \infty]{a.s} 0 . 
    \end{equation*}
\end{lemma}
\begin{proof}
    Let $K_t$ denote the conditional transition kernel of the CERW at time $t.$ Then comparing  $\Var(f(S_{t+1}) \mid \F_t) = K_tf^2(S_t)-(K_tf(S_t))^2$ with $v_f(S_t)= K^{\star}f^2(S_t)-(K^{\star}f(S_t))^2$, 
    we have  
    \begin{equation*}
        \left| \Var(f(S_{t+1})\mid \F_t)-v_f(S_t) \right| \leq \left|K_t f^2(S_t)-K^\star f^2(S_t)\right|+\left|\bigl(K_t f(S_t)\bigr)^2-\bigl(K^\star f(S_t)\bigr)^2\right|.
    \end{equation*}
    Since both $K_tf(S_t)$ and $K^{\star}f(S_t)$ are bounded in absolute value by $\|f\|_{\infty},$ we have
    \begin{align*}
        \left|\bigl(K_t f(S_t)\bigr)^2-\bigl(K^\star f(S_t)\bigr)^2\right| &\leq \left|K_t f(S_t)-K^\star f(S_t)\right|\left(\left|K_t f(S_t)\right|+\left|K^\star f(S_t)\right|\right)\\
        &\leq 4\|f\|_{\infty}^2\delta_t.
    \end{align*}
    Combining with the total variation bounds  $\left|K_t f^2(S_t)-K^\star f^2(S_t)\right| \leq 2\|f^2\|_{\infty}\delta_t = 2\| f\|_{\infty}^2\delta_t$ and         $\left|K_t f(S_t)-K^\star f(S_t)\right|\leq 2\|f\|_{\infty}\delta_t$
    we obtain 
    \begin{equation*}
        \left|\Var(f(S_{t+1}) \mid \F_t)-v_f(S_t) \right| \leq 6\|f\|_{\infty}^2\delta_t, \quad t \ge 0 . 
    \end{equation*}
    By Lemma \ref{het_lem:tv-bound} and setting $C_f=3\|f\|_{\infty}^2,$ the upper bound of this lemma holds. Almost sure convergence follows directly from  Corollary \ref{het_cor:almost-sure-vanish}.
\end{proof}
\subsection{Asymptotic quadratic variation}
\begin{proposition}\label{het_prop:quadratic-variation}
    Assume $\max_{1\le i \le k} p_i<1$. Then the quadratic variation $\langle M^{(f)} \rangle_\cdot$ satisfies 
    \begin{equation*}
        \frac{1}{n}\langle M^{(f)}\rangle_n \xrightarrow[n \to \infty]{\PP} 
        \sigma_f^2:=\pi(v_f).
    \end{equation*}
\end{proposition}
\begin{proof}
    We shall approximate the conditional variance using Lemma \ref{het_lem:variance-comparison}. By definition, 
    \begin{equation*}
        \frac{1}{n}\langle M^{(f)}\rangle_n = \frac{1}{n}\sum_{t=0}^{n-1}\Var(f(S_{t+1})|\F_t) = \frac{1}{n}\sum_{t=0}^{n-1}v_f(S_t)+ \frac{1}{n}\sum_{t=0}^{n-1}\left[\Var(f(S_{t+1})|\F_t)- v_f(S_t)\right].
    \end{equation*}
    Since the function $v_f$ is bounded, by Theorem \ref{het_thm:WLLN}, the first term converges in probability: 
        ${n}^{-1}\sum_{t=0}^{n-1}v_f(S_t) \xrightarrow{\PP}\pi(v_f)=\sigma_f^2$.
    For the second term, Lemma \ref{het_lem:variance-comparison} gives
    \begin{equation*}
        \left| \frac{1}{n}\sum_{t=0}^{n-1}\left[\Var(f(S_{t+1})|\F_t)- v_f(S_t)\right] \right| \leq \frac{C_f}{n}\sum_{t=0}^{n-1} \max_{1 \leq i \leq k} |\widehat{m}_t^{(i)}|.
    \end{equation*}
    Since $\max_{1 \leq i \leq k} |\widehat{m}_t^{(i)}|$ is bounded and vanishes almost surely,  as $t \to \infty$ by Corollary \ref{het_cor:almost-sure-vanish}, Ces\`aro's theorem implies
     ${n}^{-1}\sum_{t=0}^{n-1} \max_{1 \leq i \leq k} |\widehat{m}_t^{(i)}|\xrightarrow[n\to \infty]{a.s.} 0. $ 
    Thus, the second term converges to $0$ almost surely and hence in probability. Combining the two limits proves the result. 
\end{proof}
\begin{corollary}\label{het_col:sn}
    $s_n^2:=\E\left[ \langle M^{(f)}\rangle_n\right]$, $n \ge 1$ satisfies $\lim_{n\to \infty} s_n^2 / n = \sigma^2_f$.  If $\sigma_f^2>0,$ then  
    \begin{equation*}
\frac{\langle M^{(f)}\rangle_n}{s_n^2} =\dfrac{\langle M^{(f)}\rangle_n/n}{s_n^2/n} \xrightarrow[n \to \infty]{\PP} 1.
    \end{equation*}
\end{corollary}
\begin{proof}
    Since $\, 0 \leq \langle M^{(f)}\rangle_n /n \leq 16 \|f \|_{\infty}^2$, the sequence $n^{-1}\langle M^{(f)}\rangle_n$ is uniformly bounded and,  by Proposition \ref{het_prop:quadratic-variation}, it converges in probability to  $\sigma_f^2$. Thus, it also converges in $L^1$, namely, $\lim_{n\to \infty} {s_n^2}/{n}= \lim_{n\to \infty} \E\left[\langle M^{(f)}\rangle_n /n \right]=  \sigma_f^2$. 
    Therefore, if $\sigma_f^2>0,$ then the ratio ${\langle M^{(f)}\rangle_n}/{s_n^2}$ converges in probability. 
\end{proof} 
\subsection{Central Limit Theorem}
\begin{theorem}
    Assume $\max_{1\le i \le k} p_i<1$.  Let $f:Q_k \to \R$ be bounded. If $\sigma_f^2>0,$ then
    \begin{equation*}
        \frac{M_n^{(f)}}{s_n}\Rightarrow \mathcal{N}(0,1) \quad 
    \text{or equivalently, } \quad 
        \frac{M_n^{(f)}}{\sqrt{n}}\Rightarrow \mathcal{N}(0,\sigma_f^2), \quad \text{ as } n \to \infty.
    \end{equation*}
\end{theorem}
\begin{proof}
    We shall apply the martingale central limit theorem by verifying the two conditions of Corollary 3.1 of \cite{HallHeyde1980}. We define the triangular array $X_{n,j}:={D_j^{(f)}}/{s_n},$ for $j=1,2,\dots,n$, $n \ge 1$. Let the filtration $\F_{n,j}=\F_{j}$ for $j = 1, \ldots , n$, $n \ge 1$ for the triangular array. Since $\E[X_{n,j} \mid \F_{n,j-1}] = \E[D_j^{(f)}\mid \F_{j-1}]/s_n=0$, $(X_{n,j},\F_{n,j})_{j=1}^n$ is a martingale difference array.
    
       First, we verify the conditional variance condition. We have 
    \begin{equation*}
        \sum_{j=1}^n \E[X_{n,j}^2\mid \F_{n,j-1}]=\frac{1}{s_n^2}\sum_{j=1}^n\E[(D_j^{(f)})^2\mid \F_{j-1}]=\frac{\langle M^{(f)}\rangle_n}{s_n^2}, \quad n \ge 1. 
    \end{equation*}
    By Corollary \ref{het_col:sn}, we get the desired condition for the conditional variance. 
    
  Now we verify the conditional Lindeberg condition. Let $\varepsilon>0.$ By Proposition \ref{het_prop:doob_martingale},$|X_{n,j}|={|D_j^{(f)}|}/{s_n}\leq {4\|f\|_{\infty}}/{s_n}$.
    Since $s_n^2/n\to \sigma_f^2>0$, we have $s_n\to\infty$, as $n \to \infty$. Then for all sufficiently large $n$, ${4\|f\|_\infty}/{s_n}\leq \varepsilon$. 
    Thus, for all sufficiently large $n$, $\1_{\{|X_{n,j}|>\varepsilon\}}=0$ for every $1\leq j\leq n$.
    Consequently, $\sum_{j=1}^n \E\left[X_{n,j}^2\1_{\{|X_{n,j}|>\varepsilon\}}\mid \F_{n,j-1}\right] =0$,
    for all sufficiently large \(n\). In particular, the conditional Lindeberg condition holds.
    
      Therefore, by Corollary 3.1 of \cite{HallHeyde1980}, as $n \to \infty$
    \begin{equation*}
        \sum_{j=1}^n X_{n,j}=\frac{M_n^{(f)}}{s_n} \Rightarrow \mathcal N(0,1), 
    \end{equation*}
 and, moreover, 
    by Corollary \ref{het_col:sn} and Slutsky's theorem, ${M_n^{(f)}}/{\sqrt n} = ({M_n^{(f)}}/{s_n})({s_n}/{\sqrt n}) \Rightarrow \mathcal N(0,\sigma_f^2)$, as $n \to \infty$.   
\end{proof}
\subsection{Functional Central Limit Theorem}

\medskip

\noindent We shall show the functional central limit theorem for the rescaled process 
\begin{equation*}
    \mathcal{M}^{(f,n)}(t):= \frac{1}{\sqrt{n}}M_{\lfloor nt\rfloor}^{(f)}, \qquad t\geq 0,\,\, n \ge 1.
\end{equation*}
Since $M^{(f)}$ is a square-integrable martingale, $\mathcal{M}^{(f,n)}$ is a square-integrable c\`adl\`ag martingale with respect to the filtration $(\F_{\lfloor nt\rfloor})_{t\ge0}$. 

\begin{lemma}\label{het_lem:jump}
For every $T>0,$ the jump size $\Delta \calM^{(f,n)} (t)$ of the c\`adl\`ag martingale satisfies  
\begin{equation*}
    J(\calM^{(f,n)}, T):=\sup_{0<t\le T}|\Delta\calM^{(f,n)}(t)| \leq \frac{4\|f\|_\infty}{\sqrt n} \xrightarrow[n\to \infty]{}0.
\end{equation*}
In particular, $\lim_{n\to \infty} \E\left[J( \mathcal M^{(f,n)},T)^2\right]= 0$.   
\end{lemma}
\begin{proof}
Since the process $\mathcal{M}^{(f,n)}$ is constant in each interval $[\frac{j}{n},\frac{j+1}{n}]$, it jumps only at times $j/n$, $j \ge 1$, and
\begin{equation*}
    \Delta \mathcal M^{(f,n)}(j/n) =\frac{1}{\sqrt n} \left(M_j^{(f)}-M_{j-1}^{(f)}\right) =\frac{D_j^{(f)}}{\sqrt n}.
\end{equation*}
By the bounded increment estimate from Proposition \ref{het_prop:doob_martingale}, $\left|\Delta \mathcal M^{(f,n)}(j/n) \right| \leq {4\|f\|_{\infty}}/{\sqrt{n}}$.

Taking the supremum over $0<t \leq T$ gives
$J(\calM^{(f,n)}, T) \leq {4\|f\|_{\infty}}/{\sqrt{n}}\to 0$.
Squaring and taking expectations yields $\E\left[J( \mathcal M^{(f,n)},T)^2\right]\leq {16\|f\|_{\infty}^2}/{n} \to0$, as $n \to \infty$.
\end{proof}
\begin{lemma}\label{het_lem:quad_var_jump}
    For every \(T>0\), $n \ge 1$
\begin{equation*}
 J(\langle \mathcal M^{(f,n)}\rangle,T)
:=
\sup_{0<t\le T}
\left|
\Delta \langle \mathcal M^{(f,n)}\rangle(t)
\right|
\le
\frac{16\|f\|_\infty^2}{n}.   
\end{equation*}
In particular, $\lim_{n\to \infty} \E\left[J(\langle \mathcal{M}^{(f,n)}\rangle,T)\right]=0$.
\end{lemma}
\begin{proof}
    The proof is similar to Lemma \ref{het_lem:jump}. The predictable quadratic variation of $\langle \mathcal{M}^{(f,n)}\rangle$ is  $   \langle \mathcal M^{(f,n)}\rangle(t)={n}^{-1} \langle M^{(f)}\rangle_{\lfloor nt\rfloor}$.
    Thus its jumps occur only at time $j/n$ and 
    \begin{equation*}
        \Delta \langle \mathcal M^{(f,n)}\rangle(j/n)=\frac{1}{n}\Delta \langle M^{(f)}\rangle_j=\frac{1}{n}\E\left[(D_j^{(f)})^2\mid \F_{j-1}\right].
    \end{equation*}
    By the bounded increment estimate from Proposition \ref{het_prop:doob_martingale}, $\left|\Delta \langle \mathcal M^{(f,n)}\rangle (j/n) \right| \leq {16\|f\|_{\infty}^2}/{n}$.
    Taking the supremum over $0<t\leq T$ gives $J(\langle \mathcal M^{(f,n)}\rangle,T)\leq {16\|f\|_{\infty}^2}/{n}$.
    Therefore, 
    \begin{equation*}
        \E\left[J(\langle \mathcal{M}^{(f,n)}\rangle,T)\right]\leq \frac{16\|f\|_{\infty}^2}{n} \to0,  
    \end{equation*}
    which concludes the proof.
\end{proof}
\begin{lemma}\label{het_lem:quad-var-lim-prob}
    For every \(t\ge0\),
\begin{equation*}
\frac{1}{n}\langle M^{(f)}\rangle_{\lfloor nt\rfloor}\xrightarrow[n \to \infty]{\PP}\sigma_f^2 t.
\end{equation*}
\end{lemma}
\begin{proof}
    The case $t=0$ is immediate. Fix $t>0.$ Since $\lfloor nt\rfloor\to\infty$ as $n\to\infty,$ Proposition \ref{het_prop:quadratic-variation} gives $({\lfloor nt\rfloor})^{-1} \langle M^{(f)}\rangle_{\lfloor nt\rfloor}\xrightarrow[n\to \infty]{\PP}\sigma_f^2$.
    Also, $\frac{\lfloor nt\rfloor}{n}\to t$, as $n \to \infty$.  Therefore, we have 
\begin{equation*}
\frac1n\langle M^{(f)}\rangle_{\lfloor nt\rfloor}=\frac{\lfloor nt\rfloor}{n}\cdot\frac{1}{\lfloor nt\rfloor}\langle M^{(f)}\rangle_{\lfloor nt\rfloor}\xrightarrow[n \to \infty]{\PP}t\sigma_f^2.
\end{equation*}
\end{proof}

\begin{theorem}
    Assume $\max_{1\le i \le k} p_i<1$ and let $f:Q_k\to\R$ be a bounded function. If $\sigma_f^2>0,$ then in the c\`adl\`ag space $D([0,\infty),\R)$ the functional limit theorem holds, i.e.,  
    \begin{equation*}
        \mathcal{M}^{(f,n)}\Rightarrow \sigma_fB
    \end{equation*}
    where $B=(B_t)_{t\geq0}$ is a standard Brownian motion. 
\end{theorem}
\begin{proof}
    We shall verify that the three assumptions of condition (ii) of Theorem 2.1 of \cite{Whitt2007} hold.  
    First, by Lemma \ref{het_lem:jump},  $\lim_{n\to \infty} \E\left[J(\mathcal{M}^{(f,n)},T)^2\right]= 0$  for every $T>0.$ Thus the jumps of $\mathcal{M}^{(f,n)}$ are asymptotically negligible.
    Second, by Lemma \ref{het_lem:quad_var_jump}, $\lim_{n\to \infty} \E\left[J(\langle\mathcal{M}^{(f,n)}\rangle,T)\right]=0$  for every $T>0.$ Thus the jumps of the quadratic variation are also asymptotically negligible.
    Finally, Lemma \ref{het_lem:quad-var-lim-prob} yields, for every $t\ge0$, $\langle\mathcal M^{(f,n)}\rangle(t)={n}^{-1}\langle M^{(f)}\rangle_{\lfloor nt\rfloor}\xrightarrow[n \to \infty]{\PP}\sigma_f^2 t$.
 Therefore, all assumptions of condition {\rm(ii)} of Theorem 2.1 in \cite{Whitt2007} are satisfied with the covariance function $c(t)=\sigma_f^2 t$. The theorem implies the weak convergence to the Brownian motion $\mathcal M^{(f,n)} \Rightarrow \sigma_f B$ in $D([0,\infty),\mathbb R)$, as $n \to \infty$, where $B$ is a standard Brownian motion. 

\end{proof}

\section{Some observables on the hypercube} \label{sec:7}
The limiting variance $\sigma^2_f=\pi(v_f)$ for a bounded function $f$ or {\it observable},  where $v_f(x)=K^{\star}f^2(x)-\left(K^{\star}f(x)\right)^2$ does not depend on the parameters $p_1,p_2,\dots,p_k$ in the limit. Indeed, both objects appearing in its definition are independent of the memory parameters: $\pi$ is the uniform measure on $Q_k$ and $K^{\star}$ is the refresh kernel in \eqref{eq:refreshK}, which chooses a coordinate uniformly and refreshes it to an independent symmetric sign: 
\begin{equation} \label{eq:Kstarf}
    K^{\star}f(x)=\dfrac{1}{k}\sum_{i=1}^k\dfrac{f(x^{i\to+1})+f(x^{i\to-1})}{2}, \quad x \in Q_k . 
\end{equation}
The notation $x^{i\to s}$ was defined in Section \ref{sec:2}. We shall evaluate several limiting variances $\sigma_f^2$ for different functions $f$. The summary table is given in Table \ref{tbl:1}.

\subsection{Observable 1: Height}
Let us consider the height function 
\begin{equation*}
    H(x)=\sum_{i=1}^kx^{(i)}, \quad x \in Q_k
\end{equation*}
which represents the height level of a vertex on the hypercube. If coordinate $i$ is refreshed to $\xi$, then the new height is $H(x^{i\to\xi})=H(x)-x^{(i)}+\xi.$ Therefore, for $x \in Q_k$
\begin{equation*}
    K^{\star}H(x)=\dfrac{1}{k}\sum_{i=1}^k \E_{\xi}[H(x)-x^{(i)}+\xi]=\dfrac{1}{k}\sum_{i=1}^k \left(H(x)-x^{(i)} \right) = \dfrac{k-1}{k}H(x).
\end{equation*}
On the other hand,
\begin{equation*}
    K^{\star}H^2(x)=\dfrac{1}{k}\sum_{i=1}^k \E_{\xi}[(H(x)-x^{(i)}+\xi)^2]=\dfrac{1}{k}\sum_{i=1}^k \left[\left(H(x)-x^{(i)} \right)^2+1\right] = \left(1-\dfrac{2}{k} \right)H^2(x)+2
\end{equation*}
for $x \in Q_k$. Thus, we evaluate 
\begin{equation*}
    v_{H}(x)=\left(1-\frac{2}{k}\right)H^2(x)+2-\left(\dfrac{k-1}{k}H(x)\right)^2 = 2-\dfrac{H(x)^2}{k^2}, \quad x \in Q_k. 
\end{equation*}
Finally, averaging under the uniform measure $\pi,$ we get
\begin{equation*}
    \sigma^2_H=\pi(v_H)=2-\dfrac{1}{k^2}\E_{\pi}[H^2]= 2-\dfrac{1}{k}.
\end{equation*}
\subsection{Observable 2: Hamming distance from initial vertex}
Fix $x_0 \in Q_k$, and let 
\begin{equation*}
    D(x)=\sum_{i=1}^k\1_{\{x^{(i)}\neq x_0^{(i)}\}} , \quad x \in Q_k 
\end{equation*}
denote the Hamming distance from the initial vertex of the CERW. Notice that for each coordinate
\begin{equation*}
    \1_{\{x^{(i)}\neq x_0^{(i)}\}} = \dfrac{1-x^{(i)}x_0^{(i)}}{2} , \quad x \in Q_k . 
\end{equation*}
Thus, we can rewrite it as follows.  
\begin{equation*}
    D(x)=\dfrac{k}{2}-\dfrac{1}{2}\sum_{i=1}^kx^{(i)}x_0^{(i)} , \quad x \in Q_k.
\end{equation*}
Define the height relative to $x_0$ as
\begin{equation*}
    H_{x_0}(x):=\sum_{i=1}^kx^{(i)}x_0^{(i)} , \quad x \in Q_k. 
\end{equation*}
Notice that the function $H_{x_0}$ is the usual height applied after the coordinatewise change of variables $x^{(i)}\to x^{(i)}x_0^{(i)}$, which maps $x_0$ to $(1,1,\dots,1).$ Thus, the calculation above gives $\sigma^2_{H_{x_0}}=2-\dfrac{1}{k}.$ Since $D(x)=\frac{k}{2}-\frac{H_{x_0}(x)}{2},$
\begin{equation*}
    \sigma^2_D=\dfrac{1}{4}\sigma^2_{H_{x_0}}= \dfrac{1}{4}\left(2-\dfrac{1}{k} \right).
\end{equation*}
\subsection{Observable 3: Parity}
Let us consider the product function   
\begin{equation*}
    \Pi(x)=\prod_{i=1}^kx^{(i)} , \quad x \in Q_k 
\end{equation*}
which represents the parity of $x.$ Geometrically, $\Pi$ divides the hypercube into two bipartition classes: vertices with $\Pi(x)=1$ and vertices with $\Pi(x)=-1.$ Since adjacent vertices differ in exactly one coordinate, crossing an edge reverses the sign of $\Pi.$ If coordinate $i$ is refreshed to $\xi,$ then 
\begin{equation*}
    \Pi(x^{i\to\xi})=\xi\prod_{j\neq i} x^{(j)} , \quad x \in Q_k.
\end{equation*}
Since $\E[\xi]=0$, we have that $K^{\star}\Pi(x)=0$. On the other hand, since $\Pi(x)^2=1$ for every $x\in Q_k$, then $K^{\star}\Pi^2(x)=1.$ Therefore, $v_{\Pi}(x)=1$ and hence $\sigma^2_{\Pi}=1.$

\subsection{Observable 4: Single coordinate}
Fix a coordinate $i\in \{1,2,\dots,k\}.$ Define $P_i(x)=x^{(i)}.$ Under $K^{\star}$, the coordinate $i$ is refreshed with probability $1/k,$ and left unchanged with probability $(k-1)/k.$ If it is refreshed, its conditional mean is zero. Thus, $K^{\star}P_i(x)=\frac{k-1}{k}x^{(i)}.$ Also, $P^2_i(x)=1,$ which implies $K^{\star}P_i^2(x)=1.$ Therefore, 
\begin{equation*}
    v_{P_i}(x)=1-\left(\frac{k-1}{k}x^{(i)}\right)^2 = 1-\left(\frac{k-1}{k}\right)^2 , 
    \quad x \in Q_k
\end{equation*}
and 
\begin{equation*}
    \sigma_{P_i}^2=1-\left(\frac{k-1}{k}\right)^2 =\frac{2k-1}{k^2}.
\end{equation*}

\subsection{Observable 5: Pair correlation}
Let us consider the pairwise correlation 
\begin{equation*}
    C(x)=\sum_{i<j} x^{(i)}x^{(j)} , \quad x \in Q_k.
\end{equation*}
Since 
\begin{equation*}
    H(x)^2 = \sum_{i=1}^k (x^{(i)})^2 +2\sum_{i<j}x^{(i)}x^{(j)}=k+2C(x) , \quad x \in Q_k,
\end{equation*}
we have 
\begin{equation*}
    C(x)=\frac{H(x)^2-k}{2} , \quad x \in Q_k.
\end{equation*}
Therefore, $v_C(x)=\frac{1}{4}v_{H^2}(x)$ , $x \in Q_k$. We now compute $K^{\star}H^4(x).$ If coordinate $i$ is refreshed to $\xi,$ then
\begin{equation*}
    K^{\star}H^4(x)=\frac{1}{k}\sum_{i=1}^k\E_{\xi}[(H(x)-x^{(i)}+\xi)^4]= \left(1-\frac{4}{k} \right)H^4(x)+\left(12-\frac{16}{k} \right)H^2(x)+8
\end{equation*}
By the height calculations above, we get
\begin{align*}
    v_{H^2}(x)&=\left(1-\frac{4}{k} \right)H^4(x)+\left(12-\frac{16}{k} \right)H^2(x)+8 - \left[\left(1-\frac{2}{k}\right)H^2(x)+2  \right]^2\\ &=-\frac{4}{k^2}H^4(x) +8\left(1-\frac1k\right)H^2(x)+4 , \quad x \in Q_k. 
\end{align*}
Under $\pi,$ $\E_{\pi}[H^2]=k$ and $\E_{\pi}[H^4]=3k^2-2k.$ Hence, 
\begin{equation*}
    \sigma_{H^2}^2= \pi(v_{H^2}(x)) = 8k-16+\frac{8}{k} 
\end{equation*}
and 
\begin{equation*}
    \sigma^2_C=\frac{1}{4}\sigma_{H^2}^2=\frac{2(k-1)^2}{k}.
\end{equation*}
\subsection{Observable 6: Occupation level}
Let us consider the set 
\begin{equation*}
    \calL_r:=\{x\in Q_k: H(x)=k-2r \},
\end{equation*}
which represents the level of vertices with exactly $r$ negative coordinates and define $\ell(x) = \1_{\calL_r}(x).$ Let $R(x)$ denote the number of negative coordinates of $x.$ Then $x\in\calL_r$ if and only if $R(x)=r.$ Under one coordinate refresh, $R(x)$ can change by at most one. If $R(x)=r,$ then after refreshing a coordinate the chain remains in $\calL_r$ exactly when the refreshed coordinate keeps the same sign, which happens with probability $1/2.$ Thus,
\begin{equation*}
    K^{\star}\ell(x)=\frac{1}{2} \qquad \text{if }R(x)=r.
\end{equation*}
If $R(x)=r-1,$ then the chain enters $\calL_r$ by choosing one of the $k-r+1$ positive coordinates and refreshing it to $-1.$ Thus,
\begin{equation*}
    K^{\star}\ell(x)=\frac{k-r+1}{2k} \qquad \text{if }R(x)=r-1.
\end{equation*}
If $R(x)=r+1,$ then the chain enters $\calL_r$ by choosing one of the $r+1$ negative coordinates and refreshing it to $+1.$ Thus,
\begin{equation*}
    K^{\star}\ell(x)=\frac{r+1}{2k} \qquad \text{if }R(x)=r+1.
\end{equation*}
If $R(x)\notin\{r-1,r,r+1\}$, then one coordinate refresh cannot move the chain into $\calL_r,$ so $K^{\star}\ell(x)=0.$ We define
\begin{equation*}
    w_r(a)=\begin{cases}
       \frac{1}{2} & a=r,\\
       \frac{k-r+1}{2k} &a=r-1,\\
       \frac{r+1}{2k} &a=r+1,\\
       0 & \text{otherwise}.
    \end{cases}
\end{equation*}
Then $K^{\star}\ell(x)=w_r(R(x)).$ Since $\ell^2=\ell,$ we have 
\begin{equation*}
    v_\ell(x)=K^{\star}\ell(x)-(K^{\star}\ell(x))^2=K^{\star}\ell(x)(1-K^{\star}\ell(x)) , \quad x \in Q_k. 
\end{equation*}
Under $\pi,$ $R(x)\sim \operatorname{Bin}(k,1/2)$. Hence, we have 
\begin{equation} \label{eq:sigmaell}    \sigma_{\ell}^2=\sum_{a=0}^k\dfrac{\binom{k}{a}}{2^k}w_r(a)(1-w_r(a)).
\end{equation}
\subsection{Observable 7: Vertex indicator}
Fix $v\in Q_k,$ and let $f(x)=\1_{\{x=v\}}$ be the vertex indicator. The reference chain can reach $v$ in one step only if $x=v$ or if the Hamming distance of $x$ and $v$ is $1$. More precisely, 
\begin{equation*}
    K^{\star}(x,\{v\})= 
    \begin{cases}
       \frac{1}{2} & x=v,\\
       \frac{1}{2k} & d_H(x,v)=1,\\
       0 & d_H(x,v)\geq 2. 
    \end{cases}
\end{equation*}
Since $f=f^2,$ $v_f=K^{\star}f(x)(1-K^{\star}f(x))$. Under the invariant measure $\pi,$
\begin{equation*}
    \pi(x:d_H(x,v)=0)=\frac{1}{2^k} \qquad \text{and} \qquad \pi(x:d_H(x,v)=1)=\frac{k}{2^k}.
\end{equation*}
Therefore, 
\begin{equation*}
    \sigma_f^2=\dfrac{1}{2^k}\dfrac{1}{2}\left( 1-\frac{1}{2}\right)+ \dfrac{k}{2^k}\dfrac{1}{2k}\left( 1-\frac{1}{2k}\right) = \frac{3k-1}{4k2^k}.
\end{equation*}
\subsection{Observable 8: Face indicator}
Let 
\begin{equation*}
    F(A,\varepsilon)=\{x\in Q_k: x^{(i)}=\varepsilon_i \forall i \in A \}
\end{equation*}
where $A \subseteq \{1,2,...,k\}$ and $\varepsilon= (\varepsilon_i)_{i\in A} \in \{-1,1\}^A.$
Geometrically, $F(A,\varepsilon)$ is the face obtained by fixing the coordinates indexed by $A$ to the signs $\varepsilon_i,$ while leaving the remaining coordinates free. If $a:=|A|,$ then $F(A,\varepsilon)$ is a $(k-a)$-dimensional subcube of $Q_k,$ containing $2^{k-a}$ vertices. 
Define the face-occupation indicator
\begin{equation*}
    f(x)=\mathbf 1_{\{x\in F(A,\varepsilon)\}} , \quad x \in Q_k.
\end{equation*}
To measure how far a vertex is from satisfying the face constraints, let
\begin{equation*}
d_A(x):=\#\{i\in A:x^{(i)}\neq\varepsilon_i\} , \quad x \in Q_k.
\end{equation*}
Thus, $d_A(x)$ counts the constrained coordinates on which $x$ has the wrong sign. In other words, $d_A(x)$ is the minimum number of constrained coordinates that must be changed in order to move $x$ into the face. We will now analyze the cases:
\smallskip

If $d_A(x)=0,$ then $x\in F(A,\varepsilon)$. Refreshing a coordinate outside $A$ keeps the chain in the face, while refreshing a coordinate inside $A$ keeps it in the face with probability $1/2.$ Hence, $K^{\star}f(x)=1-\frac{a}{2k}$ for $x \in Q_k$. 
\smallskip

If $d_A(x)=1,$ then the chain enters the face only by selecting the unique incorrect coordinate and refreshing it to the correct sign. Thus,  $K^{\star}f(x)=\frac{1}{2k}$ for $x \in Q_k$. 
\smallskip

If $d_A(x)\geq 2,$ then one coordinate refresh cannot fix all incorrect coordinates, and hence $K^{\star}f(x)=0$ for $x \in Q_k$.
\smallskip

Again, since $f=f^2,$ $v_f=K^{\star}f(x)(1-K^{\star}f(x))$. Under $\pi,$ the random variable $d_A(x)$ has distribution $\operatorname{Bin}(a,1/2).$ Therefore,
\begin{equation*}
    \pi(x:d_A=0)=\frac{1}{2^a}, \qquad  \qquad \pi(x:d_A=1)=\frac{a}{2^a}
\end{equation*}
and 
\begin{equation} \label{eq:sigmafcase8}
    \sigma_f^2= \dfrac{a}{2k2^a}\left(2-\dfrac{a+1}{2k} \right).
\end{equation}

\begin{center}
\renewcommand{\arraystretch}{1.6}
\begin{table} [H]
\begin{tabular}{c|l|l|l}
\toprule
\# & Observable  & \(\sigma_f^2\) & Notes \\
\midrule
1 & $H(x)=\sum_i x^{(i)}$ & $2-\frac{1}{k}$ & Height\\
2 & $D(x)=\sum_{i} \1_{\{x^{(i)}\neq x_0^{(i)}\}}$ & $\dfrac{1}{4}\left(2-\frac{1}{k}\right)$ & Distance \\ 
3 & $\Pi(x)=\prod_ix^{(i)}$ & 1 & Parity \\
4 & $x^{(i)}$ & $\dfrac{2k-1}{k^2}$ & Single coordinate \\
5& $C(x)=\sum_{i<j}x^{(i)}x^{(j)}$ & $\dfrac{2(k-1)^2}{k}$& Pair correlation\\
6& $\1_{\{x\in \mathcal{L}_r\}}$ & See formula \eqref{eq:sigmaell} & Occupation level\\
7 & $\1_{\{x=v\}}$ & $\dfrac{3k-1}{4k2^{k}}$ & Vertex indicator\\
8& $\1_{\{x\in F(A,\varepsilon)\}}$ & See formula \eqref{eq:sigmafcase8} & Face indicator\\ \hline 
\end{tabular}
\smallskip 

\caption{Summary of observables and the corresponding asymptotic variance.}\label{tbl:1}
\end{table}
\end{center}
\vspace{-1cm}

\section{Conclusion}
In this paper, we introduced the Coordinate-wise Elephant Random Walk (CERW), a non-Markovian nearest-neighbor walk on the hypercube $Q_k=\{-1,1 \}^k$ in which each coordinate carries its own long-range memory.

The main mechanism isolated in this work is the asymptotic vanishing of coordinate-wise memory biases. When $p_i<1$ for every coordinate, each $\widehat{m}_n^{(i)}$ vanishes almost surely. As a consequence, the random transition kernels $K_n$ of the CERW become asymptotically close in total variation to the memoryless coordinate-refresh kernel $K^{\star}.$ This reduction to the refresh dynamics provides the basis for the long-time analysis: although the walk has long-range memory at finite times, its asymptotic behavior is governed by the geometry of the hypercube and the uniform measure $\pi.$

Using this perturbative comparison together with the uniform ergodicity of $K^{\star}$, we proved a weak law of large numbers for bounded observables on the hypercube. We then studied the associated Doob martingales and established a martingale central limit theorem,  as well as a martingale functional central limit theorem. For these martingale fluctuations, the limiting variance is $\sigma^2_f=\pi(v_f),$ where $\pi$ is the uniform measure on $Q_k$ and $v_f$ is the one-step variance of the observable under the refresh kernel. In particular, this martingale variance depends only on the observable $f,$ the refresh kernel $K^{\star},$ and the uniform measure $\pi,$ and is independent of the memory parameters $p_1,p_2,\dots, p_k.$

The limiting variance is therefore determined entirely by the refresh dynamics. The final section illustrates this principle by computing $\sigma^2_f$ for several natural observables on the hypercube. These examples show how the abstract martingale limit theorem translates into explicit geometric quantities on the hypercube.

\section*{Acknowledgements}
The author is grateful to Professor Tomoyuki Ichiba for his guidance, valuable discussions, and helpful comments, which greatly improved this manuscript. 

\printbibliography
\end{document}